\documentclass[twoside,a4paper,10pt]{amsart}

\usepackage{fullpage}
\usepackage{amscd}
\usepackage{amsmath}
\usepackage{amssymb}
\usepackage{amsthm}
\usepackage{dutchcal}
\usepackage{color}
\usepackage{graphicx}
\usepackage{hyperref}
\usepackage[noabbrev,capitalize]{cleveref}
\usepackage{mathrsfs}
\usepackage{mathtools}
\usepackage{pict2e}
\usepackage{stackrel}
\usepackage[T1]{fontenc}
\usepackage{tikz}
\usetikzlibrary{decorations.pathmorphing}
\usepackage[utf8]{inputenc}
\usepackage{wasysym}
\usepackage[all]{xy}
\usepackage{xcolor}
\usepackage{xy}

\hypersetup{
    colorlinks,
    linkcolor={red!50!black},
    citecolor={blue!50!black},
    urlcolor={blue!80!black}
}

\usepackage{palatino}
\allowdisplaybreaks

\newtheorem{lemma}{Lemma}[section]

\newtheorem{remark}[lemma]{Remark}

\newtheorem{theorem}[lemma]{Theorem}
\newtheorem{corollary}[lemma]{Corollary}
\newtheorem{definition}[lemma]{Definition}
\newtheorem{proposition}[lemma]{Proposition}

\newcommand{\as}{\ensuremath{\mathrm{As}}}
\newcommand{\com}{\ensuremath{\mathrm{Com}}}
\newcommand{\C}{\ensuremath{\mathscr{C}}}

\renewcommand{\P}{\ensuremath{\mathscr{P}}}

\newcommand{\SLoo}{\ensuremath{\mathcal{s}\mathscr{L}_\infty}}

\newcommand{\Ccog}{\ensuremath{\mathscr{C}\text{-}\mathsf{cog}}}

\newcommand{\Lalg}{\ensuremath{\mathcal{s}\mathscr{L}_\infty\text{-}\mathsf{alg}}}
\newcommand{\Palg}{\ensuremath{\mathscr{P}\text{-}\mathsf{alg}}}

\renewcommand{\bar}{\ensuremath{\mathrm{B}}}
\newcommand{\bara}{\ensuremath{\mathrm{B}_\alpha}}
\newcommand{\bari}{\ensuremath{\mathrm{B}_\iota}}

\newcommand{\cobar}{\ensuremath{\Omega}}
\newcommand{\cobara}{\ensuremath{\Omega_\alpha}}

\newcommand{\F}{\ensuremath{\mathcal{F}}}

\newcommand{\g}{\mathfrak{g}}
\newcommand{\h}{\mathfrak{h}}

\newcommand{\homa}{\ensuremath{\hom^\alpha}}

\newcommand{\homi}{\ensuremath{\hom^\iota}}
\newcommand{\id}{\operatorname{id}}
\renewcommand{\k}{\ensuremath{\mathbb{K}}}
\newcommand{\M}{\rotatebox[origin=c]{180}{$\mathsf{M}$}}
\newcommand{\MC}{\ensuremath{\mathrm{MC}}}
\newcommand{\op}{\ensuremath{\mathsf{Op}}}
\newcommand{\proj}{\ensuremath{\mathrm{proj}}}
\newcommand{\pr}{\ensuremath{\mathrm{pr}}}
\renewcommand{\S}{\ensuremath{\mathbb{S}}}

\newcommand{\Tw}{\ensuremath{\mathrm{Tw}}}
\DeclareMathOperator{\Ho}{Ho}

\author{Daniel Robert-Nicoud and Felix Wierstra}
\date{}
\title{Convolution algebras and the deformation theory of infinity-morphisms}
\address{Laboratoire Analyse, G\'eom\'etrie et Applications, Universit\'e Paris 13, Sorbonne Paris Cit\'e, 99 Avenue Jean Baptiste Cl\'ement, 93430 Villetaneuse, France}
\email{\href{mailto:robert-nicoud@math.univ-paris13.fr}{robert-nicoud@math.univ-paris13.fr}}
\address{Faculty of Mathematics and Physics, Charles University, Sokolovsk\'a 49/83, 186 75 Praha 8, Czech Republic}
\email{\href{mailto:felix.wierstra@gmail.com}{felix.wierstra@gmail.com}}



\thanks{The first author was supported by grants from R\'egion Ile-de-France, and the grant ANR-14-CE25-0008-01 project SAT. The second author acknowledges the financial support from the grant GA CR  No. P201/12/G028.}

\begin{document}
	
\begin{abstract}
	Given a coalgebra $C$ over a cooperad, and an algebra $A$ over an operad, it is often possible to define a natural homotopy Lie algebra structure on $\hom(C,A)$, the space of linear maps between them, called the convolution algebra of $C$ and $A$. In the present article, we use convolution algebras to define the deformation complex for $\infty$-morphisms of algebras over operads and coalgebras over cooperads. We also complete the study of the compatibility between convolution algebras and $\infty$-morphisms of algebras and coalgebras. We prove that the convolution algebra bifunctor can be extended to a bifunctor that accepts $\infty$-morphisms in both slots and which is well defined up to homotopy, and we generalize and take a new point of view on some other already known results. This paper concludes a series of works by the two authors dealing with the investigation of convolution algebras which were defined in \cite{wie16}, and \cite{rn17}, and then further studied and applied to rational homotopy theory in \cite{rnw17}.
\end{abstract}

\maketitle

\setcounter{tocdepth}{1}

\tableofcontents

\section{Introduction}

Suppose that we are given a type of algebras --- such as associative, commutative or Lie algebras, but also more elaborate ones, such as homotopy Lie or hypercommutative algebras --- and a type of coalgebras --- coassociative, cocommutative, and so on --- encoded respectively by an operad $\P$ and a cooperad $\C$. Suppose that these types of algebras and coalgebras are related by an operadic twisting morphism $\alpha:\C\to\P$. Some interesting examples of these include the universal twisting morphisms associated to the operadic bar and cobar constructions, and the twisting morphisms given by Koszul duality. Then, given a $\C$-coalgebra $C$ and a $\P$-algebra $A$, one can equip the chain complex of linear maps $\hom(C,A)$ with the structure of a (shifted) homotopy Lie algebra (usually referred to as $\SLoo$-algebras) in a canonical way. We denote this algebra by $\homa(C,A)$ and call it the \emph{convolution algebra} of $C$ and $A$.

\medskip

These algebras have already found various applications. They helped to construct a ``universal Maurer--Cartan element'' in \cite{rn17cosimplicial}, they are used to construct complete rational invariants of maps between topological spaces in \cite{wie16}, and they were applied to the construction of rational models for mapping spaces in \cite{rnw17}.

\medskip

The first of the two main results of the present paper is that one can use convolution algebras to define the correct deformation complex for $\infty$-morphisms between $\P$-algebras, as well as $\infty$-morphisms between conilpotent $\C$-coalgebras. Namely, given two $\P$-algebras, resp. conilpotent $\C$-coalgebras, we define a shifted homotopy Lie algebra whose Maurer--Cartan elements are in natural bijection with the $\infty$-morphisms between the algebras, resp. coalgebras, and which is such that two Maurer--Cartan elements are gauge equivalent if, and only if the corresponding $\infty$-morphisms are homotopic. We also relate the $\infty$-groupoid associated to the deformation complex with the mapping space in the $\infty$-category of algebras. Some partial results in this direction were already given e.g. in \cite{dol07}, \cite{dhr15}, and \cite[Sect. 7.1]{rn17}.

\medskip

Convolution algebras have been proven to behave well with respect to the tools of homotopical algebra. For example \cite[Thm. 5.1]{rn17}, they are compatible with the homotopy transfer theorem, see e.g. \cite[Sect. 10.3]{lodayvallette}. They are also compatible with a generalized notion of morphisms, called $\infty$-morphisms, in the sense that the bifunctor $\homa(-,-)$ can be extended to a bifunctor accepting $\infty$-morphisms in either one of its slots. This was proven in \cite[Prop. 4.4]{rn17} for a special case, and in full generality in \cite[Sect. 5.2]{rnw17}. Unfortunately, in \emph{op. cit.} the authors were also able to prove that one cannot perform the next natural step and extend the two bifunctors to a bifunctor accepting $\infty$-morphisms in both slots. The counterexample consists in an $\infty$-morphism $\Phi$ of $\C$-coalgebras and an $\infty$-morphism $\Psi$ of $\P$-algebras such that the two composites
\begin{equation}\label{eq:two compositions}
\homa(\Phi,1)\homa(1,\Psi)\qquad\text{and}\qquad\homa(1,\Psi)\homa(\Phi,1)
\end{equation}
are not equal, which tells us that a common extension to a new bifunctor is impossible.

\medskip

The second main result of the present article is that, assuming that the twisting morphism $\alpha$ is Koszul, these two composites are homotopic as $\infty$-morphisms of shifted homotopy Lie algebras. In particular, it is possible to extend the bifunctor $\homa(-,-)$ to accept $\infty$-morphisms in both slots if one accepts to work only up to homotopy.

\medskip

The content of this article is as follows. \cref{sect:recollections and deformation complex} begins with a short recollection of the less classical background notions we will use: those of convolution algebras and $\infty$-morphisms relative to a twisting morphism. This is followed by giving an interpretation of the Maurer--Cartan elements of a convolution algebra in terms of usual morphisms of algebras and coalgebras, and showing that two such Maurer--Cartan elements are gauge equivalent if, and only if the associated morphisms are homotopic. This is \cref{thm:MC and gauges of convolution homotopy algebras}, and it motivates the construction of a deformation complex for $\infty$-morphisms between algebras or coalgebras using convolution algebras.

\medskip

\cref{sect:main theorem and compatibility with oo-morphisms} contains most of the new results of the present article. The main result of the section is \cref{thm:fundamental thm of convolution algebras}, which describes a morphism of $\SLoo$-algebras between certain convolution algebras, and which gives us all the tools we need to study the compatibility of convolution algebras with $\infty$-morphisms. The rest of the section is mostly composed by consequences of this main theorem, and reaches its culmination with \cref{thm:compositions are homotopic}, which tells us that, even though we cannot extend the bifunctor $\homa(-,-)$ to a bifunctor taking $\infty$-morphisms in both its slots, we can do so if we pass to the homotopy categories provided the twisting morphism $\alpha$ is Koszul. The precise statement is that the two compositions described in (\ref{eq:two compositions}) are homotopic.

\medskip

Throughout the text, we postpone various technical proofs in order to improve readability. These are collected in \cref{sect:proofs}. This section also contains a result of independent interest. This is \cref{thm:comparison of deformation complexes}, which compares the deformation complex described at the end of \cref{sect:recollections and deformation complex} with another natural construction, proving that they contain exactly the same information. This last result also relates the Maurer--Cartan $\infty$-groupoid of the deformation complex with the $\infty$-categorical mapping space for algebras over an operad.

\medskip

We conclude the paper with \cref{appendix:counterexample}, where we give an explicit counterexample to the conclusion of \cref{thm:compositions are homotopic} if we remove the assumption that the twisting morphism is Koszul.

\medskip

This paper concludes a series of articles by the two authors dealing with the investigation of convolution algebras which started with \cite{wie16}, and \cite{rn17}, and then continued jointly with \cite{rnw17}.

\subsection*{Acknowledgements}

Both authors are grateful to Bruno Vallette and Alexander Berglund for their comments, advice, and constant support.

\subsection*{Notations and conventions}

We will use essentially the same notations and conventions as in \cite{rnw17}. By transitivity, we will follow the notations of the book \cite{lodayvallette} as closely as possible when talking about operads.

\medskip

We work over a field $\k$ of characteristic $0$, and over the category of chain complexes.  The dual of a chain complex will again be seen as a chain complex. All operads and cooperads in this paper are implicitly assumed to be \emph{reduced}, meaning that they are zero in arity $0$, and spanned by the identity in arity $1$. Similarly, all coalgebras and cooperads are assumed to be conilpotent.

\medskip

When talking about the homotopy theory of algebras, we always place ourselves in the Hinich model structure \cite[Thm. 4.1.1]{Hin97homological}, where the fibrations are the surjections, and the weak equivalences are the maps of algebras that are quasi-isomorphisms of the underlying chain complexes. When considering coalgebras, the model structure we will use is the one defined in \cite{dch16} generalizing \cite[Sect. 2.1]{val14}, and depends on the specific twisting morphism we are working with. If $\alpha:\C\to\P$ is a twisting morphism, the associated model structure on $\C$-coalgebras has the injections as cofibrations, and the fibrations and weak equivalences are created by the cobar functor $\cobara$. This means that a morphism of coalgebras $f$ is a fibration, resp. a weak equivalence, if, and only if $\cobara f$ is surjective, resp. a quasi-isomorphism. All Koszul morphisms induce the same model structure by \cite[Prop. 32]{legrignou16}. In the Koszul case, the class of weak equivalences is the closure of the class of filtered quasi-isomorphisms of coalgebras under the 2-out-of-3 property, see \cite[Thm. 4.9]{rn18}.

\section{\texorpdfstring{$\infty$}{Infinity}-morphisms and convolution algebras}\label{sect:recollections and deformation complex}

Since this paper is a follow-up of the article \cite{rnw17}, we will keep the recollections to a minimum and we refer the reader to \emph{op. cit.} for any need of reminders on the topic of convolution algebras or the theory surrounding them. We give nonetheless a small list of definitions and basic facts that we will need throughout the paper, and give an upgraded version of \cite[Thm. 7.1]{wie16}, see \cref{thm:MC and gauges of convolution homotopy algebras}.

\subsection{\texorpdfstring{$\infty$}{Infinity}-morphisms relative to a twisting morphism}

The notions of $\infty$-morphisms of algebras and coalgebras relative to a twisting morphism are defined as follows.

\begin{definition}
	Let $\C$ be a cooperad, let $\P$ be an operad, and let $\alpha:\C\to\P$ be a twisting morphism.\begin{enumerate}
		\item Let $A,A'$ be two $\P$-algebras. An $\infty$-morphism $\Psi$ of $\P$-algebras relative to $\alpha$ --- or and $\infty_\alpha$-morphism --- from $A$ to $A'$ is a morphism
		\[
		\Psi:\bara A\longrightarrow\bara A'
		\]
		of $\C$-coalgebras. We also write $\Psi:A\rightsquigarrow A'$.
		\item Let $C',C$ be two $\C$-coalgebras. An $\infty$-morphism $\Phi$ of $\C$-coalgebras relative to $\alpha$ --- or and $\infty_\alpha$-morphism --- from $C'$ to $C$ is a morphism
		\[
		\Psi:\cobara C'\longrightarrow\cobara C
		\]
		of $\P$-algebras. We also write $\Phi:C'\rightsquigarrow C$.
	\end{enumerate}
\end{definition}

These notions of $\infty$-morphisms relative to a twisting morphism were studied in \cite[Sect. 3]{rnw17}.

\medskip

We know that for any twisting morphism $\alpha:\C\to\P$ the relative bar construction $\bara$ preserves fibrations, and dually the relative cobar construction $\cobara$ preserves cofibrations, as they form a Quillen pair (see \cite[Thm. 3.11(1)]{dch16}, and \cite[Thm. 2.9]{val14} for the Koszul case). Therefore, any coalgebra of the form $\bara A$ is fibrant, and any algebra of the form $\cobara C$ is cofibrant. We also know that all $\P$-algebras are fibrant, and that all $\C$-coalgebras are cofibrant. Therefore, we can see $\infty_\alpha$-morphisms, both of algebras and coalgebras, as morphisms between bifibrant objects. In particular, the homotopy relation is an equivalence relation for them.

\begin{definition}
	Let $\alpha:\C\to\P$ be a twisting morphism.
	\begin{enumerate}
		\item Two $\infty_\alpha$-morphisms $A\rightsquigarrow A'$ of $\P$-algebras are \emph{homotopic} (as $\infty_\alpha$-morphisms) if they are homotopic seen as morphisms of $\C$-coalgebras $\bara A\to\bara A'$.
		\item Dually, two $\infty_\alpha$-morphisms $C'\rightsquigarrow C$ of $\C$-coalgebras are \emph{homotopic} (as $\infty_\alpha$-morphisms) if they are homotopic seen as morphisms of $\P$-algebras $\cobara C'\to\cobara C$.
	\end{enumerate}
\end{definition}

One should also compare this notion of homotopy with the results of the article \cite{dhr15}. 

\subsection{Convolution algebras}

The main subject of interest of this article are convolution algebras and their homotopical properties. We give a short reminder of how these objects appear. Recall that, given a cooperad $\C$ and an operad $\P$, there is a natural operad structure on $\hom(\C,\P)$, called the convolution operad. It was introduced in \cite[Sect. 1]{bm03}, see also \cite[Sect. 6.4.1]{lodayvallette}. We denote by $\SLoo\coloneqq\cobar\com^\vee$ the operad encoding shifted homotopy Lie algebras, see e.g. \cite[Sect. 2.7]{rnw17}.

\begin{theorem}[{\cite[Sect. 7]{wie16}}]
	Let $\C$ be a cooperad, and let $\P$ be an operad. There is a natural, canonical bijection
	\[
	\Tw(\C,\P)\cong\hom_{\op}(\SLoo,\hom(\C,\P))
	\]
	between the set of twisting morphisms from $\C$ to $\P$ and the set of morphisms of operads from the operad $\SLoo$ encoding shifted homotopy Lie algebras to the convolution operad $\hom(\C,\P)$.
\end{theorem}

This bijection is explicitly given by sending a twisting morphism $\alpha:\C\to\P$ to
\[
\M_\alpha:\SLoo\longrightarrow\hom(\C,\P)
\]
defined by $\M_\alpha(\mu_n^\vee) = \alpha(n):\C(n)\to\P(n)$. This assignment is also compatible with morphisms of operads. We refer the reader to \cite[Thm. 4.1]{rnw17} for the details.

\medskip

Let $\alpha:\C\to\P$ be a twisting morphism, let $C$ be a $\C$-coalgebra, and let $A$ be a $\P$-algebra. Then $\hom(C,A)$ is naturally a $\hom(\C,\P)$-algebra, so that we can pull its structure back along $\M_\alpha$ to get a $\SLoo$-algebra, which we denote by $\homa(C,A)$ and call the \emph{convolution algebra} of $C$ and $A$. Explicitly, if $C$ is a $\C$-coalgebra and $A$ is a $\P$-algebra, and denoting $\Delta_C:C\to\C(C)$ and $\gamma_A:\P(A)\to A$ the structure maps of $C$ and $A$ respectively, then the $\SLoo$-algebra structure of $\homa(C,A)$ is given by
\[
\ell_n(f_1,\ldots,f_n) = \gamma_A(\alpha\otimes F)^\S\Delta_C^n\ ,
\]
where $\Delta_C^n$ is the part of $\Delta_C$ landing in $(\C(n)\otimes C^{\otimes n})^{\S_n}$, and where
\[
(\alpha\otimes F)^\S\coloneqq\sum_{\sigma\in\S_n}(-1)^{\sigma(F)}\alpha\otimes f_{\sigma(1)}\otimes\cdots\otimes f_{\sigma(n)}
\]
for $f_1,\ldots,f_n\in\hom(C,A)$ and $\sigma(F)$ is the Koszul sign coming from switching around the $f_i$. The map $(\alpha\otimes F)^\S$ maps from invariants to invariants. Notice that there is an implicit identification of invariants with coinvariants before composing in $A$.

\medskip

The operation sending $(C,A)$ to $\homa(C,A)$ is compatible with morphisms of $\C$-coalgebras in the first slot, and with morphisms of $\P$-algebras in the second slot. Therefore, we obtain a bifunctor
\[
\homa:(\Ccog)^\mathrm{op}\times\Palg\longrightarrow\SLoo\text{-}\mathsf{alg}\ ,
\]
which is given by $\homa(C,A)$ on objects. Here, $\Ccog$ denotes the category of conilpotent $\C$-coalgebras, and $\Palg$ denotes the category of $\P$-algebras, both with strict morphisms.

\subsection{Maurer--Cartan elements and the deformation complex for \texorpdfstring{$\infty$}{infinity}-morphisms}

Given an $\SLoo$-al\-ge\-bra, it is natural --- from a deformation theoretical point of view --- to ask what its Maurer--Cartan elements and gauge relations are. In the case of convolution algebras, we can give a clean and complete answer.

\begin{theorem}\label{thm:MC and gauges of convolution homotopy algebras}
	Let $\alpha:\C\to\P$ be a twisting morphism, let $C$ be a $\C$-algebra, and let $A$ be a $\P$-algebra. Then there are natural bijections
	\[
	\hom_{\Ccog}(C,\bara A)\cong\MC(\homa(C,A))\cong\hom_{\Palg}(\cobara C,A)\ .
	\]
 Moreover,
	\begin{enumerate}
		\item\label{pt:1 MC and gauges} two morphisms of $\C$-coalgebras $C\to\bara A$ are homotopic if and only if the respective Maurer--Cartan elements are gauge equivalent, and
		\item\label{pt:2 MC and gauges} two morphisms of $\P$-algebras $\cobara C\to A$ are homotopic if and only if the respective Maurer--Cartan elements are gauge equivalent.
	\end{enumerate}
\end{theorem}

The proof of this result is postponed to \cref{subsect:proof of MC and gauges of convolution homotopy algebras}.

\begin{remark}
	The Maurer--Cartan equation in the result above is well defined, since all $\C$-coalgebras are supposed to be conilpotent.
\end{remark}

\begin{remark}\label{rem:prior versions of MC of convolution algebras}
	The characterization of the Maurer--Cartan set of convolution algebras was initially done in \cite[Thm. 7.1]{wie16}, where point (\ref{pt:1 MC and gauges}) is also stated, and in \cite[Thm. 6.3]{rn17}. A special case of this result can also be found in \cite[Thm. 1]{dp16}.
\end{remark}

We can use \cref{thm:MC and gauges of convolution homotopy algebras} to solve the problem of giving the correct deformation complex for both $\infty$-morphisms between $\P$-algebras, and $\infty$-morphisms between $\C$-coalgebras. The problem of defining such a deformation complex was mentioned by M. Kontsevich in his 2017 talk at S\'eminaire Bourbaki \cite{kontsevich17}. A first approach to its solution was given by \cite[Thm. 7.1]{rn17}.

\begin{definition}
	Let $\alpha:\C\to\P$ be a twisting morphism.
	\begin{enumerate}
		\item Let $A,A'$ be two $\P$-algebras. The \emph{deformation complex of $\infty_\alpha$-morphisms of $\P$-algebras} from $A$ to $A'$ is the $\SLoo$-algebra $\homa(\bara A,A')$.
		\item Let $C',C$ be two $\C$-coalgebras. The \emph{deformation complex of $\infty_\alpha$-morphisms of $\C$-coalgebras} from $C'$ to $C$ is the $\SLoo$-algebra $\homa(C',\cobara C)$.
	\end{enumerate}
\end{definition}

Indeed, a Maurer--Cartan element of $\homa(\bara A,A')$ is the same thing as a morphism $\bara A\to\bara A'$, i.e. an $\infty_\alpha$-morphism $A\rightsquigarrow A'$. Moreover, being homotopic as morphisms of $\C$-coalgebras gives an equivalence relation between $\infty_\alpha$-morphisms of $\P$-algebras, as the bar construction $\bara$ lands in the bifibrant $\C$-coalgebras. In \cite[Sect. 3.2]{val14}, it was shown that if $\alpha$ is Koszul, then this is the correct notion of homotopy equivalence for $\infty_\alpha$-morphisms. Dually, a Maurer--Cartan element of $\homa(C',\cobara C)$ is the same thing as an $\infty_\alpha$-morphism $C'\rightsquigarrow C$ of $\C$-coalgebras, and being homotopic as morphisms of $\P$-algebras is an equivalence relation on these morphisms.

\section{Compatibility between convolution algebras and \texorpdfstring{$\infty$}{infinity}-morphisms}\label{sect:main theorem and compatibility with oo-morphisms}

The main result of this section is the fact that certain natural maps between certain deformation complexes are morphisms of $\SLoo$-algebras. It has many interesting and important consequences, which we explore in \cref{subsect:two bifunctors,subsect:bifunctors commute up to homotopy}. In particular, we recover the two bifunctors from \cite[Sect. 5]{rnw17} extending $\homa(-,-)$, and we prove that they commute in the homotopy category of $\SLoo$-algebras and their $\infty$-morphisms.

\subsection{Statement of the main theorem}

Fix a twisting morphism $\alpha:\C\to\P$, and let $A,A'$ be two $\P$-algebras. Suppose we are given
\[
x\in\homa(\bara A,A')\ .
\]
Then given any $\C$-coalgebra $C$, we define a map
\[
\homa_r(1,x):\bari\homa(C,A)\longrightarrow\homa(C,A')\ ,
\]
where $\iota:\com^\vee\to\cobar\com^\vee = \SLoo$ is the natural twisting morphism. It is given as follows. Let $f_1,\ldots,f_n\in\hom(C,A)$, and let $F\coloneqq f_1\otimes\cdots\otimes f_n$ for shortness. Similarly to what done before, denote by
\[
F^\S:\left(\C(n)\otimes C^{\otimes n}\right)^{\S_n}\longrightarrow\left(\C(n)\otimes A^{\otimes n}\right)^{\S_n}
\]
the map
\[
F^\S\coloneqq\sum_{\sigma\in\S_n}(-1)^{\sigma(F)}\id_\C\otimes f_{\sigma(1)}\otimes\cdots\otimes f_{\sigma(n)}\ ,
\]
where $\sigma(F)$ is the Koszul sign obtained by switching around the $f_i$. We define $\homa_r(1,x)$ by the following diagram
\begin{center}
	\begin{tikzpicture}
		\node (a) at (0,0) {$C$};
		\node (b) at (4,0) {$\C(C)$};
		\node (c) at (4,-2) {$\left(\C(n)\otimes C^{\otimes n}\right)^{\S_n}$};
		\node (d) at (4,-4) {$\left(\C(n)\otimes A^{\otimes n}\right)^{\S_n}$};
		\node (e) at (0,-4) {$A'$};
		
		\node at (5.9,-4) {$\subset\C(A)$};
		
		\draw[->] (a) to node[above]{$\Delta_C$} (b);
		\draw[->] (b) to node[right]{$\proj_n$} (c);
		\draw[->] (c) to node[right]{$F^\S$} (d);
		\draw[->] (d) to node[above]{$x$} (e);
		\draw[->,dashed] (a) to node[left]{$\homa_r(1,x)(\mu_n^\vee\otimes F)$} (e);
	\end{tikzpicture}
\end{center}
Dually, let $C',C$ be two $\C$-coalgebras, and suppose we have $y\in\homa(C',\cobara C)$. Given a $\P$-algebra $A$, we define a map
\[
\homa_\ell(y,1):\bari\homa(C,A)\longrightarrow\homa(C',A)
\]
by sending $\mu_n^\vee\otimes F$ to the following map.
\begin{center}
	\begin{tikzpicture}
	\node (a) at (0,0) {$C'$};
	\node (b) at (4,0) {$\P(C)$};
	\node (c) at (4,-2) {$\left(\P(n)\otimes C^{\otimes n}\right)^{\S_n}$};
	\node (d) at (4,-4) {$\left(\P(n)\otimes A^{\otimes n}\right)^{\S_n}$};
	\node (e) at (0,-4) {$A$};
	
	\node at (5.9,-4) {$\subset\P(A)$};
	
	\draw[->] (a) to node[above]{$y$} (b);
	\draw[->] (b) to node[right]{$\proj_n$} (c);
	\draw[->] (c) to node[right]{$F^\S$} (d);
	\draw[->] (d) to node[above]{$\gamma_A$} (e);
	\draw[->,dashed] (a) to node[left]{$\homa_\ell(y,1)(\mu_n^\vee\otimes F)$} (e);
	\end{tikzpicture}
\end{center}
Here, we implicitly used the fact that we are working over a field of characteristic $0$ to identify invariants and coinvariants.

\medskip

The main result of this section is the following one.

\begin{theorem}\label{thm:fundamental thm of convolution algebras}
	Let $\alpha:\C\to\P$ be a twisting morphism. Let $C',C$ be two $\C$-coalgebras, and let $A,A'$ be two $\P$-algebras.
	\begin{enumerate}
		\item The map
		\[
		\homa_r(1,-):\homa(\bara A,A')\longrightarrow\homi(\bari\homa(C,A),\homa(C,A'))
		\]
		is a strict morphism of $\SLoo$-algebras.
		\item The map
		\[
		\homa_\ell(-,1):\homa(C',\cobara C)\longrightarrow\homi(\bari\homa(C,A),\homa(C',A))
		\]
		is a strict morphism of $\SLoo$-algebras.
	\end{enumerate}
\end{theorem}

The proof of this result is technical, and we postpone it to \cref{subsect: proof of fundamental thm of convolution algebras}.

\subsection{\texorpdfstring{$\infty$}{Infinity}-morphisms and two bifunctors}\label{subsect:two bifunctors}

We will now begin to unravel the consequences of \cref{thm:fundamental thm of convolution algebras}.

\medskip

The reader might have recognized the diagrams defining $\homa_r(1,-)$ and $\homa_\ell(-,1)$, as they are very similar to the ones found in \cite[Sect. 5]{rnw17}, which define two extensions of the functor $\homa(-,-)$ to take $\infty_\alpha$-morphisms in one slot or the other. We can use \cref{thm:MC and gauges of convolution homotopy algebras} to easily recover one of the main results found in \emph{loc. cit.} From now on we will identify $\infty_\alpha$-morphisms $\Psi:A\rightsquigarrow A'$ of $\P$-algebras with the associated Maurer--Cartan elements $\Psi\in\homa(\bara A,A')$, and similarly for $\infty_\alpha$-morphisms of $\C$-coalgebras.

\begin{corollary}[{\cite[Thm. 5.1]{rnw17}}]
	Let $\alpha:\C\to\P$ be a twisting morphism. Let $C',C$ be two $\C$-coalgebras, and let $A,A'$ be two $\P$-algebras.
	\begin{enumerate}
		\item Let $\Psi:A\rightsquigarrow A'$ be an $\infty_\alpha$-morphism of $\P$-algebras. Then
		\[
		\homa_r(1,\Psi):\homa(C,A)\rightsquigarrow\homa(C,A')
		\]
		is an $\infty$-morphism of $\SLoo$-algebras.
		\item Let $\Phi:C'\rightsquigarrow C$ be an $\infty_\alpha$-morphism of $\C$-coalgebras. Then
		\[
		\homa_\ell(\Phi,1):\homa(C,A)\rightsquigarrow\homa(C',A)
		\]
		is an $\infty$-morphism of $\SLoo$-algebras.
	\end{enumerate}
\end{corollary}

\begin{proof}
	We prove only the first statement, the proof of the second one being completely analogous. The $\infty_\alpha$-morphism $\Psi$ corresponds to a Maurer--Cartan element in $\homa(\bara A,A')$, which we denote again by $\Psi$ by abuse of notation. Since the map $\homa_r(1,-)$ is a morphism of $\SLoo$-algebras by \cref{thm:fundamental thm of convolution algebras}, it preserves Maurer--Cartan elements, so that $\homa_r(1,\Psi)$ is a Maurer--Cartan element of $\homi(\bari\homa(C,A),\homa(C,A'))$. But this is nothing else than to say that $\homa_r(1,\Psi)$ is an $\infty$-mor\-phism of $\SLoo$-algebras from $\homa(C,A)$ to $\homa(C,A')$, as we wanted.
\end{proof}

But \cref{thm:fundamental thm of convolution algebras} gives us even more than that.

\begin{corollary}
	Let $\alpha:\C\to\P$ be a twisting morphism. Let $C',C$ be two $\C$-coalgebras, and let $A,A'$ be two $\P$-algebras.
	\begin{enumerate}
		\item Let $\Psi,\Psi':A\rightsquigarrow A'$ be two $\infty_\alpha$-morphisms of $\P$-algebras. If $\Psi$ and $\Psi'$ are homotopic, then so are the $\infty$-morphisms of $\SLoo$-algebras $\homa_r(1,\Psi)$ and $\homa_r(1,\Psi')$.
		\item Let $\Phi,\Phi':C'\rightsquigarrow C$ be two $\infty_\alpha$-morphisms of $\C$-coalgebras. If $\Phi$ and $\Phi'$ are homotopic, then so are the $\infty$-morphisms of $\SLoo$-algebras $\homa_\ell(\Phi,1)$ and $\homa_\ell(\Phi',1)$.
	\end{enumerate}
\end{corollary}

\begin{proof}
	Again, we will only prove the first statement. The fact that $\Psi$ and $\Psi'$ are homotopic is equivalent to saying that the associated Maurer--Cartan elements of $\homa(\bara A,A')$ are gauge equivalent (by \cref{thm:MC and gauges of convolution homotopy algebras}). Since the map $\homa_r(1,-)$ is a morphism of $\SLoo$-algebras, this implies that the Maurer--Cartan elements $\homa_r(1,\Psi)$ and $\homa_r(1,\Psi')$ are also gauge equivalent, which is equivalent to say that the associated $\infty$-morphisms of $\SLoo$-algebras are homotopic.
\end{proof}

There is one important but straightforward fact that was not proven in \cite{rnw17}, which relates the morphisms $\homa_r(1,-)$ and $\homa_\ell(-,1)$, and compositions of morphisms. Recall that the action of an $\infty$-morphisms $\Theta:\g\to\h$ of $\SLoo$-algebras on Maurer--Cartan elements is given by
\[
\MC(\Theta)(x)\coloneqq\sum_{n\ge1}\frac{1}{n!}\theta_n(x,\ldots,x)\in\MC(\h)
\]
on $x\in\MC(\g)$.

\begin{proposition}\label{prop:hom and compositions}
	Let $\alpha:\C\to\P$ be a twisting morphism. Let $C',C$ be two $\C$-coalgebras, and let $A,A'$ be two $\P$-algebras.
	\begin{enumerate}
		\item Let $\Psi:A\rightsquigarrow A'$ be an $\infty_\alpha$-morphism of $\P$-algebras, and let $f:C\to\bara A$ be a morphism of $\C$-coalgebras, which we see as a Maurer--Cartan element of $\homa(C,A)$. Then
		\[
		\homa_r(1,\Psi)(f) = \left(C\stackrel{f}{\longrightarrow}\bara A\stackrel{\Psi}{\longrightarrow}\bara A'\right)
		\]
		is a morphism of $\C$-coalgebras.
		\item Let $\Phi:C'\rightsquigarrow C$ be an $\infty_\alpha$-morphism of $\C$-coalgebras, and let $g:\cobara C\to A$ be a morphism of $\P$-algebras, which we see as a Maurer--Cartan element of $\homa(C,A)$. Then
		\[
		\homa_\ell(\Phi,1)(g) = \left(\cobara C'\stackrel{\Phi}{\longrightarrow}\cobara C\stackrel{g}{\longrightarrow}A\right)
		\]
		is a morphism of $\P$-algebras.
	\end{enumerate}
\end{proposition}

\begin{proof}
	In order to give a clear proof, we will write $\widetilde{f}\in\MC(\homa(C,A))$ for the element $f$ seen as a linear map $C\to A$, and $f$ for the equivalent map of $\C$-coalgebras $C\to\bara A$. We pass from the former to the latter by
	\[
	f = (1_\C\circ\widetilde{f})\Delta_C\ .
	\]
	When writing $\Psi$, we will mean the map of $\C$-coalgebras $\bara A\to\bara A'$, and the associated Maurer--Cartan element is
	\[
	\widetilde{\Psi}\coloneqq\proj_{A'}\Psi\in\MC(\homa(\bara A,A'))\ .
	\]
	With this notation, we have
	\begin{align*}
	\MC(\homa_r(1,\widetilde{\Psi}))(\widetilde{f}) =&\ \sum_{n\ge1}\frac{1}{n!}\homa_r(1,\widetilde{\Psi})(\mu_n^\vee\otimes \widetilde{f}^{\otimes n})\\
	=& \sum_{n\ge1}\frac{1}{n!}\widetilde{\Psi}(\widetilde{f}^{\otimes n})^{\S}\Delta_C^n\\
	=& \sum_{n\ge1}\proj_{A'}\Psi(1_\C\circ \widetilde{f})\Delta_C^n\\
	=&\ \proj_{A'}\Psi(1_\C\circ \widetilde{f})\Delta_C\\
	=&\ \proj_{A'}\Psi f\ ,
	\end{align*}
	where $\mu_1^\vee = \id$. Here, $\MC(\homa_r(1,\widetilde{\Psi}))$ denotes the map induced on the set of Maurer--Cartan elements of $\homa(C,A)$ by the morphism of $\SLoo$-algebras $\homa_r(1,\widetilde{\Psi})$. The other case is similar, and left to the reader.
\end{proof}

In particular, we can take $C = \bara A''$ in \cref{prop:hom and compositions}, so that $f$ is an $\infty_\alpha$-morphism of $\P$-algebras $A''\rightsquigarrow A$, and we recover compositions of $\infty_\alpha$-morphisms. In particular, \cref{prop:hom and compositions} immediately implies that
\[
\homa_r(1,\Psi)\homa_r(1,\Psi') = \homa_r(1,\Psi\Psi')
\]
for composable $\infty_\alpha$-morphisms of $\P$-algebras, and that
\[
\homa_\ell(\Phi',1)\homa_\ell(\Phi,1) = \homa_\ell(\Phi\Phi',1)
\]
for composable $\infty_\alpha$-morphisms of $\C$-coalgebras (notice the contravariance). Thus, we recover another important result.

\begin{corollary}[{\cite[Cor. 5.4]{rnw17}}]
	The bifunctor
	\[
	\homa:(\Ccog)^\mathrm{op}\times\Palg\longrightarrow\Lalg\ ,
	\]
	extends to two bifunctors
	\[
	\homa_r:(\Ccog)^\mathrm{op}\times\infty_\alpha\text{-}\Palg\longrightarrow\Lalg\ ,
	\]
	and
	\[
	\homa_\ell:(\infty_\alpha\text{-}\Ccog)^\mathrm{op}\times\Palg\longrightarrow\Lalg\ .
	\]
\end{corollary}

\begin{proof}
	It is straightforward to check that if $f:C'\to C$ is a strict morphism of $\C$-coalgebras, and $\Psi$ is an $\infty_\alpha$-morphism of $\P$-algebras, then
	\[
	\homa(f,1)\homa_r(1,\Psi) = \homa_r(1,\Psi)\homa(f,1)\ .
	\]
	What said above then concludes the proof that $\homa_r(-,-)$ is a bifunctor. The proof for $\homa_\ell(-,-)$ is analogous. 
\end{proof}

\subsection{The two bifunctors commute up to homotopy}\label{subsect:bifunctors commute up to homotopy}

Having extended the bifunctor $\homa(-,-)$ to the two bifunctors $\homa_r(-,-)$ and $\homa_\ell(-,-)$ accepting $\infty_\alpha$-morphisms in the right and left slot respectively, it is natural to ask if those two functors admit a common extension to a bifunctor accepting $\infty_\alpha$-morphisms in both slots simultaneously. Unfortunately, this is not possible, as was proven by the authors in \cite[Sect. 6]{rnw17}. In the present paper, we will prove the next best thing.

\begin{theorem}\label{thm:compositions are homotopic}
	Let $\alpha:\C\to\P$ be a Koszul morphism. Let $C',C$ be two $\C$-coalgebras and $\Phi:C'\rightsquigarrow C$ an $\infty_\alpha$-morphism between them, and let $A,A'$ be two $\P$-algebras and $\Psi:A\rightsquigarrow A'$ an $\infty_\alpha$-morphism between them. The two composites
	\[
	\homa_\ell(\Phi,1)\homa_r(1,\Psi)\qquad\text{and}\qquad\homa_r(1,\Psi)\homa_\ell(\Phi,1)
	\]
	are homotopic as $\infty$-morphisms of $\SLoo$-algebras from $\homa(C,A)$ to $\homa(C',A')$.
\end{theorem}

The proof of this result is postponed to \cref{subsect:proof of compositions are homotopic}.

\begin{remark}
	The assumption that $\alpha$ is Koszul in this result cannot be removed --- although we do not exclude that it might be weakened. Counterexamples to the conclusion of \cref{thm:compositions are homotopic} for $\alpha$ not Koszul can easily be constructed, e.g. by taking $\alpha$ to be the zero twisting morphism. We will give such a counterexample in \cref{appendix:counterexample}.
\end{remark}

An immediate corollary is the following result.

\begin{corollary}
	Let $\alpha:\C\to\P$ be a Koszul twisting morphism. The bifunctor
	\[
	\homa:(\Ccog)^\mathrm{op}\times\Palg\longrightarrow\Lalg\ ,
	\]
	extends to a bifunctor
	\[
	\Ho(\homa):\Ho(\infty_\alpha\text{-}\Ccog)^\mathrm{op}\times\Ho(\infty_\alpha\text{-}\Palg)\longrightarrow\Ho(\infty\text{-}\Lalg)
	\]
	between the respective homotopy categories which restricts to $\homa_\ell$ and $\homa_r$ on the evident subcategories.
\end{corollary}

\section{Proof of \texorpdfstring{\cref{thm:MC and gauges of convolution homotopy algebras,thm:fundamental thm of convolution algebras,thm:compositions are homotopic}}{Theorems 2.4, 3.1 and 3.6}}\label{sect:proofs}

This section collects all of the proofs we have postponed in the rest of the article in order to improve readability.

\subsection{Proof of \cref{thm:MC and gauges of convolution homotopy algebras}}\label{subsect:proof of MC and gauges of convolution homotopy algebras}

We will prove the following theorem, which we find of independent interest. To prove \cref{thm:MC and gauges of convolution homotopy algebras} we will only need the case of the $0$th homotopy group, but proving the whole result does not require a much bigger amount of work. We denote by $\Omega_\bullet$ the simplicial commutative algebra given by the polynomial de Rham forms on the simplices, and for a $\SLoo$-algebra $\g$ be denote by
\[
\MC_\bullet(\g)\coloneqq\MC(\g\otimes\Omega_\bullet)
\]
the simplicial set of Maurer--Cartan elements. Notice that this definition does not make sense if one does not assume some kind of ``finiteness'' condition on $\g$, since the Maurer--Cartan equation is an infinite sum of elements in a chain complex. The appropriate condition in this context is requiring that $\g$ carries a sensible filtration with respect to which it is complete, see e.g. \cite[Def. 8.1]{rnw17}. Since in the present paper we work exclusively with conilpotent coalgebras, all the convolution algebras appearing automatically satisfy this condition, as we will see shortly.

\begin{theorem}\label{thm:comparison of deformation complexes}
	Let $\alpha:\C\to\P$ be a twisting morphism, let $C$ be a $\C$-coalgebra, and let $A$ be a $\P$-algebra. Then we have a natural homotopy equivalence of simplicial sets
	\[
	\MC_\bullet(\homa(C,A))\simeq\MC(\homa(C,A\otimes\Omega_\bullet))\ .
	\]
	induced by the canonical inclusion
	\[
	\Theta:\hom(C,A)\otimes \Omega_\bullet\stackrel{\cong}{\longrightarrow}\hom(C,A\otimes \Omega_\bullet)
	\]
	given on pure tensors by sending $\phi\otimes\omega$ with $\phi\in\hom(C,A)$ and $\omega\in\Omega_\bullet$ to
	\[
	\Theta(\phi\otimes\omega)=\big(c\in C\longmapsto\phi(c)\otimes\omega\big)\ .
	\]
\end{theorem}

Here, the $\SLoo$-algebra $\homa(C,A)$ is filtered by
\[
\F_n\homa(C,A)\coloneqq\left\{\phi\in\hom(C,A)\ \vline\ \F_n^\mathrm{corad}C\subseteq\ker\phi\right\},
\]
where $\F_\bullet^\mathrm{corad}C$ is the coradical filtration of $C$, see e.g. \cite[Sect. 5.8.4]{lodayvallette}, which is exhaustive since we supposed that all coalgebras are conilpotent. This makes $\homa(C,A)$ into a complete $\SLoo$-algebra, so that it makes sense to define $\MC_\bullet(\homa(C,A))$. A similar filtration is put on $\homa(C,A\otimes\Omega_n)$, for all $n\ge0$. Complete $\SLoo$-algebra were called \emph{filtered} $\SLoo$-algebras in \cite[Def. 8.1]{rnw17}.

\medskip

One would like to go the easy way, and to prove the statement simply by saying that $\homa(C,A\otimes\Omega_\bullet)$ is isomorphic to $\homa(C,A)\otimes\Omega_\bullet$. However, since $\Omega_{\bullet}$ is infinite dimensional this is not true unless $C$ is finite dimensional. We can work around this problem as follows. There is a contraction due to Dupont \cite{dup76}
\begin{center}
	\begin{tikzpicture}
	\node (a) at (0,0){$\Omega_\bullet$};
	\node (b) at (2,0){$C_\bullet$};
	
	\draw[->] (a)++(.3,.1)--node[above]{\mbox{\tiny{$p_\bullet$}}}+(1.4,0);
	\draw[<-,yshift=-1mm] (a)++(.3,-.1)--node[below]{\mbox{\tiny{$i_\bullet$}}}+(1.4,0);
	\draw[->] (a) to [out=-150,in=150,looseness=4] node[left]{\mbox{\tiny{$h_\bullet$}}} (a);
	\end{tikzpicture}
\end{center}
from $\Omega_\bullet$ to a simplicial sub-complex $C_\bullet$. This sub-complex can be thought of as the dual of the cellular complex of the geometric simplices. In particular, it is finite dimensional in every simplicial degree. The reader desiring more detail should consult the original article \cite{dup76}, or e.g. \cite[Sect. 3]{get09}. Then, we proceed as follows:
\begin{enumerate}
	\item\label{pt:1 proof we thm} We transfer the simplicial $\SLoo$-algebra structure from $\homa(C,A\otimes\Omega_\bullet)$ to $\hom(C,A\otimes C_\bullet)$ using the Dupont contraction, and prove that the simplicial sets obtained from these algebras by taking the Maurer--Cartan elements are homotopy equivalent.
	\item\label{pt:2 proof we thm} We prove that the simplicial $\SLoo$-algebra $\hom(C,A\otimes C_\bullet)$ thus obtained is isomorphic to the simplicial $\SLoo$-algebra $\hom(C,A)\otimes C_\bullet$ similarly obtained from $\homa(C,A)\otimes\Omega_\bullet$ by homotopy transfer theorem.
	\item\label{pt:3 proof we thm} We conclude by using some results of \cite{rn17cosimplicial} to prove that the simplicial sets of Maurer--Cartan elements of $\hom(C,A)\otimes C_\bullet$ and of $\homa(C,A)\otimes\Omega_\bullet$ are homotopy equivalent.
\end{enumerate}
The proof of point (\ref{pt:1 proof we thm}) uses methods very similar to the ones used to prove \cite[Thm. 3.3]{rn17cosimplicial}, whose demonstration is itself inspired from the proof of the Dolgushev--Rogers theorem \cite{dr15}, and which implies point (\ref{pt:3 proof we thm}).

\medskip

Now to make the details more precise. The Dupont contraction induces a contraction
\begin{center}
	\begin{tikzpicture}
	\node (a) at (0,0){$\homa(C,A\otimes\Omega_\bullet)$};
	\node (b) at (5,0){$\hom(C,A\otimes C_\bullet)$};
	
	\draw[->] (a)++(1.4,.1)--node[above]{\mbox{\tiny{$(1_A\otimes p_\bullet)_*$}}}+(2.3,0);
	\draw[<-,yshift=-1mm] (a)++(1.4,-.1)--node[below]{\mbox{\tiny{$(1_A\otimes i_\bullet)_*$}}}+(2.3,0);
	\draw[->] (-1.4,-.1) to [out=-150,in=150,looseness=10] node[left]{\mbox{\tiny{$(1_A\otimes h_\bullet)_*$}}} (-1.4,.1);
	\end{tikzpicture}
\end{center}
By the homotopy transfer theorem --- e.g. \cite[Sect. 10.3]{lodayvallette} --- we can endow $\hom(C,A\otimes C_\bullet)$ with a $\SLoo$-algebra structure, and extend $(1_A\otimes p_\bullet)_*$ and $(1_A\otimes i_\bullet)_*$ to simplicial $\infty_\iota$-morphisms of $\SLoo$-algebra, where as usual $\iota:\com^\vee\to\SLoo$ is the natural twisting morphism. We denote by
\[
P_\bullet:\MC(\homa(C,A\otimes\Omega_\bullet))\longrightarrow\MC_\bullet(\homa(C,A))
\]
and
\[
I_\bullet:\MC_\bullet(\homa(C,A))\longrightarrow\MC(\homa(C,A\otimes\Omega_\bullet))
\]
the morphisms of simplicial sets induced by these $\infty$-morphisms on the Maurer--Cartan sets. The following result implies point (\ref{pt:1 proof we thm}).

\begin{proposition}
	The morphisms of simplicial sets $P_\bullet$ and $I_\bullet$ are homotopy inverses  to each other.
\end{proposition}

\begin{proof}
	The proof of \cite[Thm. 3.3]{rn17cosimplicial} goes through essentially unchanged by replacing $\g\otimes\Omega_\bullet$ by $\homa(C,A\otimes\Omega_\bullet)$, and $\g\otimes C_\bullet$ by $\hom(C,A\otimes C_\bullet)$. Therefore, we will only give a sketch of the proof here, and refer to \emph{op. cit.} for the details.
	
	\medskip
	
	First of all, the composite $P_\bullet I_\bullet$ is the identity, cf. \cite[Lemma 3.5]{rn17cosimplicial}. Therefore, it is enough to prove that the composite
	\[
	R_\bullet\coloneqq I_\bullet P_\bullet:\MC(\homa(C,A\otimes\Omega_\bullet))\longrightarrow\MC(\homa(C,A\otimes\Omega_\bullet))
	\]
	is a weak equivalence. This is done by an inductive procedure on the filtration, and then passing to the limit.
	
	\medskip
	
	The map $R_0$ is simply given by the identity, and thus induces a bijection at the level of the $0$th homotopy group. Then one considers the case where $C=\F_2^\mathrm{corad}C$, where $\homa(C,-)$ lands in abelian $\SLoo$-algebras, and proves that $R_\bullet$ is a weak equivalence of simplicial sets in that case, with the same methods as \cite[Lemma 3.7]{rn17cosimplicial}. Further, one notices that \cite[Lemmas 3.8 and 3.10]{rn17cosimplicial} also hold in this context, so that all of the arguments carry over to the present situation, giving us the result we wanted.
\end{proof}

The other contraction we consider is
\begin{center}
	\begin{tikzpicture}
	\node (a) at (0,0){$\homa(C,A)\otimes\Omega_\bullet$};
	\node (b) at (5,0){$\hom(C,A)\otimes C_\bullet$};
	
	\draw[->] (a)++(1.4,.1)--node[above]{\mbox{\tiny{$1_{\hom(C,A)}\otimes p_\bullet$}}}+(2.3,0);
	\draw[<-,yshift=-1mm] (a)++(1.4,-.1)--node[below]{\mbox{\tiny{$1_{\hom(C,A)}\otimes i_\bullet$}}}+(2.3,0);
	\draw[->] (-1.4,-.1) to [out=-150,in=150,looseness=10] node[left]{\mbox{\tiny{$1_{\hom(C,A)}\otimes h_\bullet$}}} (-1.4,.1);
	\end{tikzpicture}
\end{center}
Again, the homotopy transfer theorem gives us a $\SLoo$-algebra structure on $\hom(C,A)\otimes C_\bullet$. Since $C_\bullet$ is finite dimensional in every simplicial degree, we have a natural isomorphism of simplicial chain complexes
\[
\widetilde{\Theta}:\hom(C,A)\otimes C_\bullet\stackrel{\cong}{\longrightarrow}\hom(C,A\otimes C_\bullet)
\]
given on pure tensors by sending $\phi\otimes\omega$ with $\phi\in\hom(C,A)$ and $\omega\in C_\bullet$ to
\[
\widetilde{\Theta}(\phi\otimes\omega)=\big(c\in C\longmapsto\phi(c)\otimes\omega\big)\ .
\]
We check that $\widetilde{\Theta}$ respects the transferred $\SLoo$-structures.

\begin{proposition}\label{prop:isom of the two transfered structures}
	The map $\widetilde{\Theta}$ is an isomorphism of $\SLoo$-algebras with respect to the two transferred structures.
\end{proposition}

\begin{proof}
	We consider the diagram
	\begin{center}
		\begin{tikzpicture}
		\node (a) at (0,0){$\homa(C,A)\otimes\Omega_\bullet$};
		\node (b) at (5,0){$\hom(C,A)\otimes C_\bullet$};
		
		\draw[->] (a)++(1.4,.1)--node[above]{\mbox{\tiny{$1_{\hom(C,A)}\otimes p_\bullet$}}}+(2.3,0);
		\draw[<-,yshift=-1mm] (a)++(1.4,-.1)--node[below]{\mbox{\tiny{$1_{\hom(C,A)}\otimes i_\bullet$}}}+(2.3,0);
		\draw[->] (-1.4,-.1) to [out=-150,in=150,looseness=10] node[left]{\mbox{\tiny{$1_{\hom(C,A)}\otimes h_\bullet$}}} (-1.4,.1);
		
		\node (c) at (0,-2.5){$\homa(C,A\otimes\Omega_\bullet)$};
		\node (d) at (5,-2.5){$\hom(C,A\otimes C_\bullet)$};
		
		\draw[->] (c)++(1.4,.1)--node[above]{\mbox{\tiny{$(1_A\otimes p_\bullet)_*$}}}+(2.3,0);
		\draw[<-,yshift=-1mm] (c)++(1.4,-.1)--node[below]{\mbox{\tiny{$(1_A\otimes i_\bullet)_*$}}}+(2.3,0);
		\draw[->] (-1.4,-2.6) to [out=-150,in=150,looseness=10] node[left]{\mbox{\tiny{$(1_A\otimes h_\bullet)_*$}}} (-1.4,-2.4);
		
		\draw[->] (a) to node[left]{$\Theta$} (c);
		\end{tikzpicture}
	\end{center}
	We have that $\Theta$ preserves the filtrations, and
	\[
	\Theta(1_{\hom(C,A)}\otimes h_\bullet) = (1_A\otimes h_\bullet)_*\Theta\ .
	\]
	Therefore, the morphism $\Theta$ is a morphism of complete contractions in the sense of \cite[Def. 1.7]{ban17}. Since $\Theta$ is a morphism of $\SLoo$-algebras, then by \cite[Lemma 1.10]{ban17} we have that
	\[
	\pr_2(\Theta)\coloneqq (1_{\hom(C,A)}\otimes p_\bullet)\Theta(1_A\otimes i_\bullet)_*
	\]
	is a morphism of $\SLoo$-algebras between the transfered $\SLoo$-algebras. But
	\begin{align*}
		(1_{\hom(C,A)}\otimes p_\bullet)\Theta(1_A\otimes i_\bullet)_*(\phi\otimes\omega) =&\ (1_{\hom(C,A)}\otimes p_\bullet)\Theta(\phi\otimes i(\omega))\\
		=&\ (1_{\hom(C,A)}\otimes p_\bullet)\big(c\mapsto\phi(c)\otimes i(\omega)\big)\\
		=&\ \big(c\mapsto\phi(c)\otimes pi(\omega)\big)\\
		=&\ \big(c\mapsto\phi(c)\otimes\omega\big)\\
		=&\ \widetilde{\Theta}(\phi\otimes\omega)
	\end{align*}
	for any $\phi\in\hom(C,A)$ and $\omega\in C_\bullet$. This concludes the proof.
\end{proof}

Now we can conclude the proof of the theorem.

\begin{proof}[Proof of \cref{thm:comparison of deformation complexes}]
	We have taken care of points (\ref{pt:1 proof we thm}) and (\ref{pt:2 proof we thm}) in our program. Point (\ref{pt:3 proof we thm}) is given by \cite[Thm. 3.3]{rn17cosimplicial}. Putting these steps together, we obtain the desired homotopy equivalence.
	
	\medskip
	
	To prove that $\MC(\Theta)$ is indeed a homotopy equivalence, we refine slightly the proof of \cref{prop:isom of the two transfered structures} and notice that the second part of \cite[Lemma 1.10]{ban17}, together with the fact that
	\[
	(1_{\hom(C,A)}\otimes p_\bullet)_\infty(1_{\hom(C,A)}\otimes i_\bullet)_\infty = 1_{\hom(C,A)\otimes C_\bullet}\ ,
	\]
	implies that
	\[
	P_\bullet\MC(\Theta)I_\bullet = \MC(\widetilde{\Theta})\ .
	\]
	Since all the morphisms except --- \emph{a priori} --- $\MC(\Theta)$ are homotopy equivalences, we conclude that $\MC(\Theta)$ must also be one, concluding the proof of the theorem.
\end{proof}

As a consequence, we have the following.

\begin{proof}[Proof of \cref{thm:MC and gauges of convolution homotopy algebras}]
	As mentioned in \cref{rem:prior versions of MC of convolution algebras}, the proof of the fact that Maurer--Cartan elements corresponds to morphisms of algebras and coalgebras was done in  \cite[Thm. 7.1(1)]{wie16} and in a special case in \cite[Thm. 6.3]{rn17}.
	
	\medskip
	
	We are left to prove points (\ref{pt:1 MC and gauges}) and (\ref{pt:2 MC and gauges}) of the statement. For this, notice that if $A$ is a $\P$-algebra, then $A\otimes\Omega_1$ is a good path object for $A$. This is seen by considering the two morphisms of $\P$-algebras
	\[
	A\longrightarrow A\otimes\Omega_1\longrightarrow A\times A\ ,
	\]
	the first one sending $a\in A$ to $a\otimes1\in A\otimes\Omega_1$ and the second one sending $a(t_0,t_1)\in A\otimes\Omega_1$ to $(a(1,0),a(0,1))\in A\times A$. Therefore, two morphisms of $\P$-algebras
	\[
	f_0,f_1:\cobara C\longrightarrow A
	\]
	are homotopic if and only if there exists a morphism of $\P$-algebras
	\[
	H:\cobara C\longrightarrow A\otimes\Omega_1
	\]
	such that $H$ equals $f_0$, respectively $f_1$, if we evaluate its image at $(1,0)$, resp. $(0,1)$. But this is exactly saying that
	\[
	H\in\homa(C,A\otimes\Omega_1)
	\]
	is a $1$-simplex going from $f_0\in\homa(C,A)$ to $f_1\in\homa(C,A)$, i.e. that $f_0$ and $f_1$ represent the same element of $\pi_0\MC(\homa(C,A))$. Finally, by \cref{thm:comparison of deformation complexes} this is equivalent to the fact that $f_0$ and $f_1$ are in the same path component of $\MC_\bullet(\homa(C,A))$, i.e. that they are gauge equivalent. This concludes the proof of point (\ref{pt:2 MC and gauges}).
	
	\medskip
	
	The proof of point (\ref{pt:1 MC and gauges}) is done analogously using the fact that $\bara(A\otimes\Omega_1)$ is a good path object for $\bara A$, which we show now. The bar functor $\bara$ is a right Quillen functor by \cite[Thm. 3.11(1)]{dch16}, or by \cite[Thm. 2.1.3]{val14} in the case where $\alpha$ is Koszul. Thus, it preserves fibrations, and by Ken Brown's lemma --- see e.g. \cite[Lemma 1.1.12]{Hovey} --- it also sends weak equivalences between fibrant objects to weak equivalences. But all $\P$-algebras are fibrant, and $\bara$ also preserves limits being right adjoint, so it follows that $\bara$ preserves good path objects. Therefore, $\bara(A\otimes\Omega_1)$ is a good path object for $\bara A$, and one concludes the proof of point (\ref{pt:1 MC and gauges}) proceeding in the same way as before.
\end{proof}

\subsection{Proof of \cref{thm:fundamental thm of convolution algebras}}\label{subsect: proof of fundamental thm of convolution algebras}

We begin by proving a technical result that we will need in the main proofs.

\begin{lemma}\label{lemma:technical lemma 2}
	Let $\C$ be a cooperad, and let $C$ be a conilpotent $\C$-coalgebra. Then
	\[
	(\Delta_{(1)}\circ1_C)\Delta_C^n = \sum_{\substack{n_1+n_2=n+1\\1\le i\le n_1}}\left(1_\C\circ\left(1_C^{\otimes(i-1)}\otimes\Delta_C^{n_2}\otimes1_C^{\otimes(n-i)}\right)\right)\Delta_C^{n_1}\ .
	\]
	Moreover, let $f_1,\ldots,f_n\in\hom(C,V)$ for a chain complex $V$. Under the canonical inclusion
	\[
	\bigoplus_{\substack{n_1+n_2 = n+1\\1\le i\le n_1}}\C(n_1)\circ\left(C^{\otimes(i-1)}\otimes\left(\C(n_2)\otimes C^{\otimes n_2}\right)\otimes C^{\otimes(n-i)}\right)\xhookrightarrow{\qquad}(\C\circ\C)(n)\otimes C^{\otimes n}
	\]
	we have
	\[
	F^\S(\Delta_{(1)}\circ1_C)\Delta_C^n = \sum_{\substack{S_1\sqcup S_2=[n]}}(-1)^\epsilon((F^{S_1}\Delta_C^{n_1})\otimes F^{S_2})\Delta_C^{n_2}\ ,
	\]
	where $n_1 = |S_1|$ and $n_2 = |S_2|+1$, and again $\epsilon$ is given by the Koszul sign rule because we are shuffling the maps $f_i$.
\end{lemma}

\begin{proof}
	For the first identity, one considers the equality
	\[
	(\Delta_\C\circ 1_C)\Delta_C = (1_\C\circ\Delta_C)\Delta_C
	\]
	and then projects on the subspace
	\[
	(\C\circ_{(1)}\C)(C)\cong \C\circ(C;\C(C))\ .
	\]
	We leave the details to the reader. The second statement then follows in a straightforward way.
\end{proof}

We can now prove the main result.

\begin{proof}[Proof of \cref{thm:fundamental thm of convolution algebras}]
	We prove the first case, the second one being dual. We will begin by proving that the map $\homa_r(1,-)$ commutes with the brackets, and then show that it also commutes with the differentials. All of this will be done by explicitly writing down and comparing the formul{\ae}.
	
	\medskip
	
	Let $x_1,\ldots,x_k\in\homa(\bara A,A')$. Then we have
	\[
	\ell_k(x_1,\ldots,x_k) = \gamma_{A'}(\alpha\otimes X)^\S(\Delta_\C^k\circ1_A)\ ,
	\]
	where $X\coloneqq x_1\otimes\cdots\otimes x_k$, and thus for $f_1,\ldots,f_n\in\homa(C,A)$ we get
	\begin{align*}
	\homa_r(1,\ell_k(X))(\mu_n^\vee&\otimes F) = \ell_k(X)F^\S\Delta_C^n\\
	=&\ \gamma_{A'}(\alpha\otimes X)^\S(\Delta_\C^k\circ1_A)F^\S\Delta_C^n\\
	=&\ \gamma_{A'}(\alpha\otimes X)^\S F^\S(\Delta_\C^k\circ1_C)\Delta_C^n\\
	=&\ \sum_{n_1+\cdots+n_k=n}\gamma_{A'}(\alpha\otimes X)^\S F^\S(1_\C\circ(\Delta_C^{n_1}\otimes\cdots\otimes\Delta_C^{n_k}))\Delta_C^k\\
	=&\ \sum_{S_1\sqcup\cdots\sqcup S_k=[n]}(-1)^{\epsilon_1}\gamma_{A'}(\alpha\otimes X)^\S(1_\C\circ(F^{S_1}\Delta_C^{n_1}\otimes\cdots\otimes F^{S_k}\Delta_C^{n_k}))\Delta_C^k\\
	=&\ \sum_{\substack{S_1\sqcup\cdots\sqcup S_k=[n]\\\sigma\in\S_k}}(-1)^{\epsilon_1+\epsilon_2}\gamma_{A'}(\alpha\circ1_{A'})(1_\C\circ(x_{\sigma(1)}F^{S_1}\Delta_C^{n_1}\otimes\cdots \otimes x_{\sigma(k)}F^{S_k}\Delta_C^{n_k}))\Delta_C^k
	\end{align*}
	where the fourth line follows from $(\Delta_\C\circ C)\Delta_C = (1_\C\circ\Delta_C)\Delta_C$, and in the passage from the fourth line to the fifth line we denoted $n_i\coloneqq|S_i|$. The Koszul signs are
	\[
	\epsilon_1 = \sum_{i=1}^k\sum_{s\in S_i}|f_s|\sum_{j<i}\sum_{\substack{p\in S_j\\p>s}}|f_p|\ ,
	\]
	obtained by shuffling the maps $f_i$, and $\epsilon_2$, which is similarly obtained by permuting the maps $x_i$ and interchanging them with the maps $f_j$.
	
	\medskip
	
	On the other hand, we have
	\begin{align*}
	\ell_k(\homa(1,X))(\mu_n^\vee\otimes F) =&\ \big(\gamma_{\homa(C,A')}(\iota\otimes\homa(1,X))^\S(\Delta_{\com^\vee}^k\otimes1_{\hom(C,A)})\big)(\mu_n^\vee\otimes F)\\
	=&\ \big(\gamma_{\homa(C,A')}(\iota\otimes\homa(1,X))^\S\big)\left(\sum_{\substack{S_1\sqcup\ldots\sqcup S_k = [n]}}(-1)^{\epsilon_1}\mu_k^\vee\otimes\bigotimes_{i=1}^k(\mu_{n_i}^\vee\otimes F^{S_i})\right)\\
	=&\ \gamma_{\homa(C,A')}\left(\sum_{\substack{S_1\sqcup\ldots\sqcup S_k = [n]\\\sigma\in\S_k}}(-1)^{\epsilon_1+\epsilon_2}s^{-1}\mu_k^\vee\otimes\bigotimes_{i=1}^k\homa(1,x_{\sigma(i)})(\mu_{n_i}^\vee\otimes F^{S_i})\right)\\
	=&\ \sum_{\substack{S_1\sqcup\ldots\sqcup S_k = [n]\\\sigma\in\S_k}}(-1)^{\epsilon_1+\epsilon_2}\gamma_{A'}\left(\alpha\otimes\bigotimes_{i=1}^k\homa(1,x_{\sigma(i)})(\mu_{n_i}^\vee\otimes F^{S_i})\right)\Delta_C^k\\
	=&\ \sum_{\substack{S_1\sqcup\ldots\sqcup S_k = [n]\\\sigma\in\S_k}}(-1)^{\epsilon_1+\epsilon_2}\gamma_{A'}\left(\alpha\otimes\bigotimes_{i=1}^kx_{\sigma(i)}F^{S_i}\Delta_C^{n_i}\right)\Delta_C^k\ .
	\end{align*}
	Notice that in the fourth line we do not need to sum over permutations when applying the structure map $\gamma_{\homa(C,A')}$, because the term
	\[
	\sum_{\substack{S_1\sqcup\ldots\sqcup S_k = [n]\\\sigma\in\S_k}}(-1)^{\epsilon_1+\epsilon_2}s^{-1}\mu_k^\vee\otimes\bigotimes_{i=1}^k\homa(1,x_{\sigma(i)})(\mu_{n_i}^\vee\otimes F^{S_i})
	\]
	in the third line naturally lives in invariants, not coinvariants.
	
	\medskip
	
	In conclusion, we have
	\[
	\homa(1,\ell_k(X)) = \ell_k(\homa(1,X))\ .
	\]
	We are left to prove that the morphism commutes with the differentials. In order to avoid cumbersome notation, we will denote simply by $d$ the differentials of both the $\SLoo$-algebras $\homa(\bara A,A')$ and $\homi(\bari\homa(C,A),\homa(C,A'))$. The letter $\partial$ denotes instead the differential of $\hom(C,A)$. Let $x\in\homa(\bara A,A')$, and let $f_1,\ldots,f_n\in\homa(C,A)$. On one hand, we have
	\begin{align*}
	d(x) =&\ d_{A'}x - (-1)^{|x|}x d_{\bara A}\\
	=&\ d_{A'}x - (-1)^{|x|}x d_{\C\circ A} - (-1)^{|x|}x (1_\C\circ(1_A;\gamma_A))((1_\C\circ_{(1)}\alpha)\circ1_A)(\Delta_{(1)}\circ1_A)\ ,
	\end{align*}
	and thus
	\begin{align*}
	\homa(1,d(x))(\mu_n^\vee\otimes F) =&\ d(x)F^\S\Delta_C^n\\
	=&\ d_{A'}x F^\S\Delta_C^n - (-1)^{|x|}x d_{\C\circ A}F^\S\Delta_C^n\tag{L1}\label{L1}\\
	&- (-1)^{|x|}x (1_\C\circ(1_A;\gamma_A))((1_\C\circ_{(1)}\alpha)\circ1_A)(\Delta_{(1)}\circ1_A)F^\S\Delta_C^n\ .\tag{L2}\label{L2}
	\end{align*}
	The second term in (\ref{L1}) equals
	\begin{align*}
	xd_{\C\circ A}F^\S\Delta_C^n =&\ x(d_\C\circ1_A)F^\S\Delta_C^n + x(1_\C\circ'd_A)F\Delta_C^n\\
	=&\ (-1)^{|F|}x F^\S(d_\C\circ1_C)\Delta_C^n + x\partial(F)^\S\Delta_C^n + (-1)^{|F|} xF(1_\C\circ'd_C)\Delta_C^n\\
	=&\ (-1)^{|F|}x F^\S d_{\C(C)}\Delta_C^n + x\partial(F)^\S\Delta_C^n\\
	=&\ (-1)^{|F|}x F^\S\Delta_C^nd_C + x\partial(F)^\S\Delta_C^n\ ,\tag{\text{T1}}\label{T1}
	\end{align*}	
	while the term of (\ref{L2}) gives
	\begin{align*}
	x(1_\C\circ(1_A;\gamma_A))&((1_\C\circ_{(1)}\alpha)\circ1_A)(\Delta_{(1)}\circ1_A)F^\S\Delta_C^n =\\
	=&\ x (1_\C\circ(1_A;\gamma_A))((1_\C\circ_{(1)}\alpha)\circ1_A)F^\S(\Delta_{(1)}\circ1_C)\Delta_C^n\\
	=&\ \sum_{S_1\sqcup S_2=[n]}(-1)^\epsilon x (1_\C\circ(1_A;\gamma_A))((1_\C\circ_{(1)}\alpha)\circ1_A)((F^{S_1}\Delta_C^{n_1})\otimes F^{S_2})\Delta_C^{n_2}\ ,\tag{\text{T2}}\label{T2}
	\end{align*}
	where in the third line we used \cref{lemma:technical lemma 2}. On the other hand,
	\begin{align*}
	d(\homa(1,x)) =&\ d_{\homa(C,A')}\homa(1,x) - (-1)^{|\homa(1,x)|}\homa(1,x)d_{\bari\homa(C,A)}\ .
	\end{align*}
	Notice that $(-1)^{|\homa(1,x)|} = (-1)^{|x|}$. We apply this to $\mu_n^\vee\otimes F$ and obtain
	\begin{align*}
	d(\homa(1,x))(\mu_n^\vee\otimes F) =&\ d_{A'}\homa(1,x) - (-1)^{|x|+|F|}\homa(1,x)(\mu_n^\vee\otimes F)d_C\\
	&-(-1)^{|x|}\homa(1,x)(d_{\bari\homa(C,A)}(\mu_n^\vee\otimes F))\ .
	\end{align*}
	The first term equals $d_{A'}x F\Delta_C^n$ and cancels with the first term of (\ref{L1}), and the second term equals the first term of (\ref{T1}). For the third term, we have
	\begin{align*}
	\homa(1,x)(d_{\bari\homa(C,A)}&(\mu_n^\vee\otimes F)) =\\
	=&\ \homa(1,x)(\mu_n^\vee\otimes \partial(F))\\
	&+ \homa(1,x)\left(\sum_{S_1\sqcup S_2=[n]}\mu_{n_2}^\vee\otimes(\ell_{n_1}(F^{S_1})\otimes F^{S_2})\right).\tag{T3}\label{T3}
	\end{align*}
	The first term of this expression cancels the second term of (\ref{T1}). Therefore, we are left to show that (\ref{T3}) equals (\ref{T2}). We have
	\begin{align*}
	\homa(1,x)&\left(\sum_{S_1\sqcup S_2=[n]}(-1)^\epsilon\mu_{n_2}^\vee\otimes(\ell_{n_1}(F^{S_1})\otimes F^{S_2})\right) =\\
	=&\ \sum_{S_1\sqcup S_2=[n]}(-1)^\epsilon x(\ell_{n_1}(F^{S_1})\otimes F^{S_2})\Delta_C^{n_2}\\
	=&\ \sum_{S_1\sqcup S_2=[n]}(-1)^\epsilon x((\gamma_A(\alpha\otimes F^{S_1})\Delta_C^{n_1})\otimes F^{S_2})\Delta_C^{n_2}\\
	=&\ \sum_{S_1\sqcup S_2=[n]}(-1)^\epsilon x(1_\C\circ(\gamma_A(\alpha\circ1_A)\otimes1_A^{\otimes (n_2-1)}))(1_\C\circ(F^{S_1}\Delta_C^{n_1}\otimes F^{S_2}))\Delta_C^{n_2}\\
	=&\ (\ref{T2})\ ,
	\end{align*}
	where $n_1 = |S_1|$, and $n_2 = |S_2|+1$. This concludes the proof.
\end{proof}

\subsection{Proof of \cref{thm:compositions are homotopic}}\label{subsect:proof of compositions are homotopic}

The strategy to prove that the two compositions are homotopic as $\infty$-morphisms of $\SLoo$-algebras, or equivalently gauge equivalent as Maurer--Cartan elements, is as follows. We know that when $\Psi$ would be a strict morphism, then the two compositions would commute. Therefore, we will rectify $\Psi$ to get a strict morphism $R(\Psi)$ of $\P$-algebras, which we do by applying the bar-cobar adjunction. However, to do this we need also to rectify the $\P$-algebras $A$ and $A'$. Fortunately, if $\alpha$ is Koszul, then the new $\P$-algebras thus obtained are quasi-isomorphic, and thus we are able to apply the Dolgushev--Rogers theorem \cite[Thm. 1.1]{dr15} to go back to the original two compositions and conclude the proof.

\begin{proof}[Proof of \cref{thm:compositions are homotopic}]
	Denote by $R(A)\coloneqq\cobara\bara A$ the bar-cobar resolution of $A$, and similarly for $A'$. Since $\alpha$ is Koszul, the counit of the adjunction
	\[
	\epsilon_A:R(A)\longrightarrow A
	\]
	is a quasi-isomorphism by \cite[Thm. 11.3.3]{lodayvallette}. The rectification of the $\infty_\alpha$-morphism $\Psi$ is given by the strict morphism
	\[
	R(\Psi):\cobara\bara A\xrightarrow{\cobara\Psi}\cobara\bara A'\ .
	\]
	The proof is outlined by the diagram in \cref{f:diagram}.
	\begin{figure*}[tbp]
		\begin{center}
			\begin{tikzpicture}
			\node (a) at (0,0){$\homa(C,A)$};
			\node (b) at (6,6){$\homa(C',A)$};
			\node (c) at (8,0){$\homa(C',R(A'))$};
			\node (d) at (4,0){$\homa(C,R(A))$};
			\node (e) at (6,3.5){$\homa(C',R(A))$};
			\node (f) at (6,-3.5){$\homa(C,R(A'))$};
			\node (g) at (6,-6){$\homa(C,A')$};
			\node (h) at (12,0){$\homa(C',A')$};
			
			\draw[->,line join=round,decorate,decoration={zigzag,segment length=4,amplitude=.9,post=lineto,post length=2pt}] (a) to [out = 90, in = -180] node[above,sloped]{$\homa(\Phi,1)$} (b);
			\draw[->,line join=round,decorate,decoration={zigzag,segment length=4,amplitude=.9,post=lineto,post length=2pt}] (a) to [out = -90, in = 180] node[below,sloped]{$\homa(1,\Psi)$} (g);
			\draw[<-] (a) -- node[above]{\scriptsize $\homa(1,\epsilon_A)$} node[below]{\scriptsize filtered qi} (d);
			\draw[<-] (b) -- node[above,sloped]{\scriptsize $\homa(1,\epsilon_A)$} node[below,sloped]{\scriptsize filtered qi} (e);
			\draw[<-] (g) -- node[above,sloped]{\scriptsize $\homa(1,\epsilon_{A'})$} node[below,sloped]{\scriptsize filtered qi} (f);
			\draw[->,line join=round,decorate,decoration={zigzag,segment length=4,amplitude=.9,post=lineto,post length=2pt}] (d) to [out = 90, in = -155] node[above,sloped]{\small $\homa(\Phi,1)$} (e);
			\draw[->,line join=round,decorate,decoration={zigzag,segment length=4,amplitude=.9,post=lineto,post length=2pt}] (f) to [out = 25, in = -90] node[below,sloped]{\small $\homa(\Phi,1)$} (c);
			\draw[->,line join=round,decorate,decoration={zigzag,segment length=4,amplitude=.9,post=lineto,post length=2pt}] (b) to [out = 0, in = 90] node[above,sloped]{$\homa(1,\Psi)$} (h);
			\draw[->,line join=round,decorate,decoration={zigzag,segment length=4,amplitude=.9,post=lineto,post length=2pt}] (g) to [out = 0, in = -90] node[below,sloped]{$\homa(\Phi,1)$} (h);
			\draw[->] (d) to [out = -90, in = 155] node[below,sloped]{\small $\homa(1,R(\Psi))$} (f);
			\draw[->] (e) to [out = -25, in = 90] node[above,sloped]{\small $\homa(1,R(\Psi))$} (c);
			\draw[<-] (h) -- node[above]{\scriptsize $\homa(1,\epsilon_{A'})$} node[below]{\scriptsize filtered qi} (c);
			\end{tikzpicture}
		\end{center}
		\caption{}
		\label{f:diagram}
	\end{figure*}
	The innermost square is commutative since $R(\Psi)$ is a strict morphism of $\P$-algebras, and the maps passing from the outer rim to the inner one are filtered quasi-isomorphisms. Notice that all squares are commutative, except for the outer one, which fails to be commutative at $\homa(C,A)$.
	
	\medskip
	
	Now consider the morphism of $\SLoo$-algebras
	\[
	\homi(\bari\homa(C,A),\homa(C',A'))\xrightarrow{\homi(\bari\homa(1,\epsilon_A),1)}\homi(\bari\homa(C,R(A)),\homa(C',A'))\ .
	\]
	It is a filtered quasi-isomorphism, and it is given on Maurer--Cartan elements by precomposition with $\homa(1,\epsilon_A)$. The two compositions
	\[
	\homa(\Phi,1)\homa(1,\Psi)\quad\text{and}\quad\homa(1,\Psi)\homa(\Phi,1)
	\]
	are naturally elements of $\homi(\bari\homa(C,A),\homa(C',A'))$ and are mapped to the same elements, and thus, by the Dolgushev--Rogers theorem, they are homotopic.
\end{proof}

\begin{remark}
	The proof above supposes that we are filtering our convolution algebras with the filtration induced by a filtration on the $\C$-coalgebras --- usually the coradical filtration. If one filters them by a filtration induced by filtrations on the $\P$-algebras, then the exact same proof goes through with the sole difference that one has to rectify the $\infty_\alpha$-morphism $\Phi$ instead of $\Psi$.
\end{remark}

\appendix

\section{A counterexample}\label{appendix:counterexample}

The goal of this appendix is to give an explicit counterexample to the conclusion of \cref{thm:compositions are homotopic} in the case when the twisting morphism $\alpha$ is not Koszul.

\medskip

More precisely, we will show that there exists a (non-Koszul) twisting morphism $\alpha:\C\to\P$, two $\C$-coalgebras $C',C$, an $\infty_\alpha$-morphism $\Phi:C'\rightsquigarrow C$ of $\C$-coalgebras, two $\P$-algebras $A,A'$, and an $\infty_\alpha$-morphism $\Psi:A\rightsquigarrow A'$ which are such that the two composites
\[
\homa(\Phi,1)\homa(1,\Psi)\quad\text{and}\quad\homa(1,\Psi)\homa(\Phi,1)
\]
are not homotopic.

\medskip

For simplicity, we will work in the non-symmetric setting. It is straightforward to construct a version of the example we present in the symmetric setting.

\medskip

We take $\alpha:\as^\vee\to\as$ to be the zero twisting morphism $\alpha = 0$. This greatly simplifies the situation, because $\infty_0$-morphisms are very simple.
\begin{enumerate}
	\item If $A$ is an associative algebra, then $\bar_0A = \as^\vee(A)$ with the differential $d_{\as^\vee(A)} = 1_{\as^\vee}\circ'd_A$ induced only by the differential of $A$. Therefore, an $\infty_0$-morphism of associative algebras $A\rightsquigarrow A'$ is just a chain map $\as^\vee(A)\to A'$.
	\item Dually, an $\infty_0$-morphism of coassociative coalgebras $C'\rightsquigarrow C$ is nothing else than a chain map $C'\to\as(C)$.
\end{enumerate}
Moreover, we will take our (co)algebras to be concentrated in degree $0$, and thus having trivial differential. Thus, we end up working with linear maps between vector spaces, and two such maps are homotopic if, and only if, they are equal.

\medskip

The coassociative coalgebra $C'$ will be
\[
C'\coloneqq\k x\qquad\text{with trivial coproduct, i.e.}\qquad\Delta_{C'}(x)\coloneqq\id\otimes x\in\as^\vee(C')\ .
\]
For the coassociative coalgebra $C$ we take
\[
C\coloneqq\as^\vee(\k y)\ ,
\]
the cofree coassociative coalgebra over a $1$-dimensional vector space. The $\infty_0$-morphism $\Phi:C'\rightsquigarrow C$ is given by the linear map
\[
\Phi:C'\longrightarrow\as(C)
\]
defined by
\[
\Phi(x)\coloneqq \mu_2\otimes\big((\mu_2^\vee\otimes y\otimes y)\otimes(\id\otimes y)\big)\ ,
\]
where $\mu_n\in\as(n)$ is the operation corresponding to the multiplication of $n$ elements in an associative algebra.

\medskip

For the algebras, we set $A$ to be
\[
A\coloneqq\as(\k z)\ ,
\]
the free associative algebra on one generator, and $A'$ to be
\[
A'\coloneqq\k w\qquad\text{with}\qquad\mu_2(w,w)\coloneqq w\ .
\]
The $\infty_0$-morphism $\Psi:A\rightsquigarrow A'$ is given by any linear map
\[
\Psi:\as^\vee(A)\longrightarrow A'
\]
satisfying
\[
\Psi(\id\otimes(\id\otimes z)) = w\qquad\text{and}\qquad\Psi\left(\mu_2^\vee\otimes\big((\id\otimes z)\otimes(\id\otimes z)\big)\right) = w\ .
\]
For example, one can define $\Psi$ by the conditions above and setting it to be zero everywhere else.

\medskip

Finally, we take $f\in\hom^0(C,A)$ to be any linear map such that $f(\id\otimes y) = \id\otimes z$. For example, one can simply set $f$ to be $0$ on all other elements. We will consider the action of the two compositions on the element $\mu_3^\vee\otimes F\coloneqq\mu_3^\vee\otimes f\otimes f\otimes f$, and then apply the resulting map to $x\in C'$. We refer to \cite[Sect. 6.1]{rnw17} for a diagrammatic description of the two composites. In formul{\ae}, we have that
\begin{align*}
	\hom^0(1,\Psi)&\hom^0(\Phi,1)(\mu_3^\vee\otimes F)(x) = \Psi\as^\vee(\gamma_A)F^\S\proj_3\as^\vee(\Phi)\Delta_{C'}(x)\\
	=&\ \Psi\as^\vee(\gamma_A)F^\S\proj_3\as^\vee(\Phi)(\id\otimes x)\\
	=&\ \Psi\as^\vee(\gamma_A)F^\S\proj_3\left(\id\otimes\left(\mu_2\otimes\big((\mu_2^\vee\otimes y\otimes y)\otimes(\id\otimes y)\big)\right)\right)\\
	= 0\ ,
\end{align*}
since the element in the second to last line lives in $(\as^\vee\circ\as)(2)\otimes C^{\otimes 2}$. At the same time, the other composition gives
\begin{align*}
	\hom^0(\Phi,1)&\hom^0(1,\Psi)(\mu_3^\vee\otimes F)(x) = \gamma_{A'}\as(\Psi)F^\S\proj_3\as(\Delta_C)\Phi(x)\\
	=&\ \gamma_{A'}\as(\Psi)F^\S\proj_3\as(\Delta_C)\left(\mu_2\otimes\big((\mu_2^\vee\otimes y\otimes y)\otimes(\id\otimes y)\big)\right)\\
	=&\ \gamma_{A'}\as(\Psi)F^\S\proj_3\Big(\mu_2\otimes\left(\big((\id\otimes(\mu_2^\vee\otimes y\otimes y))\otimes(\id\otimes(\id\otimes y))\right)\big) +\\
	&\quad\qquad\qquad\qquad\qquad + \big((\mu_2^\vee\otimes((\id\otimes y)\otimes(\id\otimes y)))\otimes(\id\otimes(\id\otimes y))\big)\Big)\\
	=&\ \gamma_{A'}\as(\Psi)F^\S\left(\mu_2\otimes\big((\mu_2^\vee\otimes((\id\otimes y)\otimes(\id\otimes y)))\otimes(\id\otimes(\id\otimes y))\big)\right)\\
	=&\ \gamma_{A'}\as(\Psi)\left(\mu_2\otimes\big((\mu_2^\vee\otimes((\id\otimes z)\otimes(\id\otimes z)))\otimes(\id\otimes(\id\otimes z))\big)\right)\\
	=&\ \gamma_{A'}\left(\mu_2\otimes w\otimes w\big)\right)\\
	=&\ w\ .
\end{align*}
Thus, the compositions are not equal, and in particular they are not homotopic.

\bibliographystyle{alpha}
\bibliography{Convolution_algebras_and_the_deformation_theory_of_infinity_morphisms}

\begin{thebibliography}{RNW17}

\bibitem[Ban17]{ban17}
R.~Bandiera.
\newblock Descent of {D}eligne--{G}etzler $\infty$-groupoids.
\newblock 2017.
\newblock \href{https://arxiv.org/abs/1705.02880}{arXiv:1705.02880}.

\bibitem[BM03]{bm03}
C.~Berger and I.~Moerdijk.
\newblock Axiomatic homotopy theory for operads.
\newblock {\em Commentarii Mathematici Helvetici}, 78(4):805--831, 2003.
\newblock \href{https://arxiv.org/abs/math/0206094}{arXiv:math/0206094}.

\bibitem[DCH16]{dch16}
G.~C. Drummond-Cole and J.~Hirsh.
\newblock Model structures for coalgebras.
\newblock {\em Proceedings of the AMS}, 144(4):1467--1481, 2016.
\newblock \href{https://arxiv.org/abs/1411.5526}{arXiv:1411.5526}.

\bibitem[DHR15]{dhr15}
V.~A. Dolgushev, A.~E. Hoffnung, and C.~L. Rogers.
\newblock What do homotopy algebras form?
\newblock {\em Advances in Mathematics}, 274:562--605, 2015.
\newblock \href{https://arxiv.org/abs/1406.1751}{arXiv:1406.1751}.

\bibitem[Dol07]{dol07}
V.~A. Dolgushev.
\newblock Erratum to: ``{A} proof of {T}sygan's formality conjecture for an
  arbitrary smooth manifold''.
\newblock 2007.
\newblock \href{https://arxiv.org/abs/math/0703113}{arXiv:math/0703113}.

\bibitem[DP16]{dp16}
V.~Dotsenko and N.~Poncin.
\newblock A tale of three homotopies.
\newblock {\em Applied Categorical Structures}, 24(6):845--873, 2016.

\bibitem[DR15]{dr15}
V.~A. Dolgushev and C.~L. Rogers.
\newblock A version of the {G}oldman-{M}illson theorem for filtered
  {$L_\infty$}-algebras.
\newblock {\em Journal of Algebra}, 430:260--302, 2015.
\newblock \href{https://arxiv.org/abs/1407.6735}{arXiv:1407.6735}.

\bibitem[Dup76]{dup76}
J.~L. Dupont.
\newblock Simplicial de {R}ham cohomology and characteristic classes of flat
  bundles.
\newblock {\em Topology}, 15:233--245, 1976.

\bibitem[Get09]{get09}
E.~Getzler.
\newblock Lie theory for nilpotent {$\L_\infty$}-algebras.
\newblock {\em Annals of Mathematics}, 170(1):271--301, 2009.
\newblock \href{https://arxiv.org/abs/math/0404003}{arXiv:math/0404003}.

\bibitem[Hin97]{Hin97homological}
V.~Hinich.
\newblock Homological algebra of homotopy algebras.
\newblock {\em Communications in Algebra}, 25(10):3291--3323, 1997.
\newblock \href{https://arxiv.org/abs/q-alg/9702015}{arXiv:q-alg/9702015}.

\bibitem[Hov99]{Hovey}
M.~Hovey.
\newblock {\em Model {C}ategories}, volume~63 of {\em Mathematical Surveys and
  Monographs}.
\newblock AMS, 1999.

\bibitem[Kon17]{kontsevich17}
M.~Kontsevich.
\newblock Derived {G}rothendieck--{T}eichm{\"u}ller group and graph complexes
  (after {T}. {W}illwacher).
\newblock {\em Report of the {S}{\'e}minaire {N}icolas {B}ourbaki},
  2016--2017(1126), 2017.

\bibitem[LG16]{legrignou16}
B.~Le~Grignou.
\newblock Homotopy theory of unital algebras.
\newblock 2016.
\newblock \href{https://arxiv.org/abs/1612.02254}{arXiv:1612.02254}.

\bibitem[LV12]{lodayvallette}
J.~L. Loday and B.~Vallette.
\newblock {\em Algebraic Operads}, volume 346 of {\em Grundlehren der
  {M}athematischen {W}issenschaften}.
\newblock Springer Verlag, 2012.

\bibitem[RN17]{rn17cosimplicial}
D.~Robert-Nicoud.
\newblock Representing the {D}eligne--{H}inich--{G}etzler $\infty$-groupoid.
\newblock 2017.
\newblock \href{https://arxiv.org/abs/1702.02529}{arXiv:1702.02529}.

\bibitem[RN18a]{rn17}
D.~Robert-Nicoud.
\newblock Deformation theory with homotopy algebra structures on tensor
  products.
\newblock {\em Documenta Mathematica}, 23:189--240, 2018.
\newblock \href{https://arxiv.org/abs/1702.02194}{arXiv:1702.02194}.

\bibitem[RN18b]{rn18}
D.~Robert-Nicoud.
\newblock A model structure for the {G}oldman--{M}illson theorem.
\newblock {\em Graduate Journal of Mathematics}, 3(1):15--30, 2018.
\newblock \href{https://arxiv.org/abs/1803.03144}{arXiv:1803.03144}.

\bibitem[RNW17]{rnw17}
D.~Robert-Nicoud and F.~Wierstra.
\newblock Homotopy morphisms between convolution homotopy {L}ie algebras.
\newblock 2017.
\newblock \href{https://arxiv.org/abs/1712.00794}{arXiv:1712.00794}.

\bibitem[Val14]{val14}
B.~Vallette.
\newblock Homotopy theory of homotopy algebras.
\newblock 2014.
\newblock \href{https://arxiv.org/abs/1411.5533}{arXiv:1411.5533}.

\bibitem[Wie16]{wie16}
F.~Wierstra.
\newblock Algebraic {H}opf invariants and rational models for mapping spaces.
\newblock 2016.
\newblock \href{https://arxiv.org/abs/1612.07762}{arXiv:1612.07762}.

\end{thebibliography}

\end{document}